\documentclass[12pt]{amsart}
\usepackage{amssymb,enumerate,amsthm,yfonts,mathrsfs}
\usepackage{amsmath}
\usepackage{bbm}           
\usepackage{bm}

\usepackage[colorlinks=true, pdfstartview=FitV, linkcolor=blue, citecolor=blue, urlcolor=blue]{hyperref}

\makeatletter
\def\Ddots{\mathinner{\mkern1mu\raise\p@
\vbox{\kern7\p@\hbox{.}}\mkern2mu
\raise4\p@\hbox{.}\mkern2mu\raise7\p@\hbox{.}\mkern1mu}}
\makeatother

\def\Xint#1{\mathchoice
{\XXint\displaystyle\textstyle{#1}}%
{\XXint\textstyle\scriptstyle{#1}}%
{\XXint\scriptstyle\scriptscriptstyle{#1}}%
{\XXint\scriptscriptstyle\scriptscriptstyle{#1}}%
\!\int}
\def\XXint#1#2#3{{\setbox0=\hbox{$#1{#2#3}{\int}$}
\vcenter{\hbox{$#2#3$}}\kern-.5\wd0}}

\def\dashint{\Xint-}

\newtheorem{theorem}{Theorem}[section]

\theoremstyle{definition}

\def\al{{\alpha}}
\def\R{\mathbb R}

\def\Z{\mathbb Z}

\def\bey{\begin{eqnarray*}}
\def\eey{\end{eqnarray*}}
 
 \allowdisplaybreaks

\def\D{{\mathscr D}}

\def\Sp{{\mathscr S}}

\def\B{{\mathscr B}}

\begin{document}

\title[Sharp Weighted Bounds]{Sharp weighted bounds without testing or extrapolation}
\author{Kabe Moen}


 \address{Department of Mathematics, University of Alabama, Tuscaloosa, AL 35487-0350 }
\email{kmoen@as.ua.edu}



\thanks{The author is partially supported by the NSF under grant 1201504}
 
\maketitle

\begin{abstract} We give a short proof of the sharp weighted bound for sparse operators that holds for all $p$, $1<p<\infty$.  By recent developments this implies the bounds hold for any Calder\'on-Zygmund operator.  The novelty of our approach is that we avoid two techniques that are present in  other proofs: two weight inequalities and extrapolation.  Our techniques are applicable to fractional integral operators as well.

\end{abstract}

\section{Introduction} In this note we give an accessible proof of the sharp weighted bound for Calder\'on-Zygmund operators of Hyt\"onen \cite{Hy2}: 
\begin{equation}\label{sharpCZ}\|T\|_{\B(L^p(w))}\leq c_{p,T} [w]_{A_p}^{\max(1,\frac{p'}{p})}\qquad 1<p<\infty.\end{equation}
Inequality \eqref{sharpCZ} was originally termed the $A_2$ theorem because if it held for $p=2$ then it held for all $p$, $1<p<\infty$; a fact that followed from a sharp version of Rubio de Francia's extrapolation theorem.  Most known proofs of \eqref{sharpCZ} follow this paradigm, beginning with $p=2$ and then extrapolating \cite{CMP,Hy2,HP,LPR,Le6,Pet2,PV}.  Of these proofs, some of them use two weight testing inequalities \cite{Hy2,HP,LPR} while others do not \cite{CMP,Pet2,PV}. Some proofs avoid extrapolation, but use two weight testing inequalities \cite{HL,HLMORSU,HLP}.  Our proof of \eqref{sharpCZ} holds for all $1<p<\infty$ and does not use two weight testing inequalities or extrapolation.  Moreover, we do not use Bellman functions as the original proofs in \cite{Pet2,PV} do.  Our techniques can be thought of as an extension of those in \cite{CMP} to $p\not=2$ and were inspired by the mixed estimates in \cite{Le3}.  The methods also apply to fractional integral operators; as a result, we are able to give a new proof of the bound from \cite{LMPT},
\begin{equation}\label{sharpFrac}\|I_\al\|_{\B(L^p(w^p),L^q(w^q))}\leq c_{p,\al} [w]_{A_{p,q}}^{(1-\frac{\al}{n})\max(1,\frac{p'}{q})},\end{equation}
for a range of $p$ and $q$ that satisfy $\frac1q=\frac1p-\frac{\al}{n}$.  Each of the proofs uses the same tools: dyadic operators, sparse families of cubes, and universal maximal function bounds.  
\section{Preliminaries}

We will consider two central families of integral operators in harmonic analysis.  A Calder\'on-Zygmund operator will be an $L^2(\R^n)$ bounded linear operator, associated to a standard kernel $K$ (see Grafakos \cite[p. 171]{GrMo}), that has the representation
$$Tf(x):=\int_{\R^n}K(x,y)f(y)\,dy$$
for $f\in L^2_c(\R^n)$ and $x\notin\text{supp}\ f$.
The family of fractional integral operators or Riesz potentials is defined for $0<\al<n$, by
$$I_\al f(x):=\int_{\R^n}\frac{f(y)}{|x-y|^{n-\al}}\,dy.$$

Calder\'on-Zygmund operators are bounded on $L^p(w)$ for $1<p<\infty$ when $w\in A_p$:
$$[w]_{A_p}:=\sup_{Q}\left(\dashint_Qw(x)\,dx\right)\left(\dashint_Q w(x)^{1-p'}\,dx\right)^{p-1}<\infty.$$
Meanwhile, when $1<p<n/\al$ and $1/q=1/p-\al/n$, the fractional integral operator, $I_\al$, maps
$L^p(w^p)$ into $L^q(w^q)$ exactly when $w\in A_{p,q}$:
$$[w]_{A_{p,q}}:=\sup_Q\left(\dashint_Qw(x)^q\,dx\right)\left(\dashint_Q w(x)^{-p'}\,dx\right)^{\frac{q}{p'}}<\infty.$$

A {\it dyadic grid}, denoted $\D$, is a collection of cubes in $\R^n$ with the following properties:
\begin{enumerate}[(i)] 
\item each $Q\in \D$ satisfies $|Q|=2^{nk}$ for some $k\in \Z$; 
\item if $Q,P\in \D$ then $Q\cap P=\varnothing,P,$ or $Q$;
\item for each $k\in \Z$, the family $\D_k=\{Q\in \D:|Q|=2^{nk}\}$ forms a partition of $\R^n$.
\end{enumerate}

We say that a family of dyadic cubes $\Sp\subset \D$ is {\it sparse} if for each $Q\in \Sp$, 
$$\Big|\bigcup_{\substack{Q'\in \Sp\\ Q'\subsetneq Q}} Q'\Big|\leq \frac12|Q|.$$
Given a sparse family, $\Sp$, if we define
$$E(Q):=Q\, \backslash \bigcup_{\substack{Q'\in \Sp \\ Q'\subsetneq Q}} Q',$$
then the family $\{E(Q)\}_{Q\in \Sp}$ is pairwise disjoint,  $E(Q)\subset Q$, and $|Q|\leq 2|E(Q)|$.  Sparse families have long been used in Calder\'on-Zygmund theory, our definition can be found in \cite{HN}.

We now introduce discrete versions of $T$ and $I_\al$ using sparse families of cubes.  We begin with the simpler operator $I_\al$.  Given a dyadic grid $\D$ and $0<\al<n$, define the dyadic fractional integral operator as
$$I^\D_\al f:=\sum_{Q\in \D} |Q|^{\frac{\al}{n}}\dashint_Q f\,dx\cdot \chi_Q.$$
In \cite{CM} it is proven that there exists a finite collection of dyadic grids $\D^1,\ldots, \D^N$ such that
\begin{equation} \label{dyadictononI} I_\al f\simeq \max_{1\leq t\leq N} I_\al^{\D^t}f.\end{equation}
We may further decompose the operator $I_\al^\D$ using sparse families.  If $f\geq 0$ is bounded with compact support, then there exists a sparse family depending on $f$, $\Sp=\Sp(f)$, such that
\begin{equation}\label{sparseI} I^\D_\al f\simeq \sum_{Q\in \Sp}|Q|^{\frac{\al}{n}}\dashint_Q  f\,dx\cdot \chi_Q:=I_\al^\Sp f\end{equation}
where the implicit constants depend on the dimension and $\al$, but not on $f$ or $\Sp$.  The equivalence \eqref{sparseI} can be found in \cite{CM,LMPT,Per2,SW}.   By combining \eqref{dyadictononI} with \eqref{sparseI} we obtain the following theorem.
\begin{theorem} \label{dyadicbound} If $X$ and $Y$ are Banach function spaces (see \cite[Chapter 1]{BS}) then there exists a constant $c_\al$ such that
$$ \|I_\al\|_{\B(X,Y)}\leq c_\al \sup_{\Sp\subset \D}\|I_\al^\Sp\|_{\B(X,Y)}.$$
\end{theorem}

To define a dyadic version of a Calder\'on-Zygmund operator it is not enough to simply take $\al=0$ in the definition of $I^\D_\al$, because the defining series will not always converge.  Dyadic versions of Calder\'on-Zygmund operators must capture the cancelation of the operator $T$; one way to accomplish this is by using the so called Haar shift operators \cite{Hy2}.  Another way is to use families of sparse cubes.  Let $\Sp\subset \D$ be a sparse family and define sparse Calder\'on-Zygmund operator
$$T^{\Sp}f:=\sum_{Q\in \Sp}\dashint_Q f\,dx\cdot\chi_Q.$$
Lerner \cite{Le5,Le6} proved the corresponding version of Theorem \ref{dyadicbound} for Calder\'on-Zygmund operators.
\begin{theorem} \label{dyadicboundCZ} Suppose $X$ and $Y$ are Banach function spaces on $\R^n$ and $T$ is a Calder\'on-Zygmund operator, then there exists a constant $c_T$
$$\|T\|_{\B(X,Y)}\leq c_T\sup_{\Sp\subset \D}\|T^\Sp\|_{\B(X,Y)}.$$
\end{theorem}

Finally we will need one more tool: universal maximal function bounds.  Given a Borel measure $\mu$ on $\R^n$ define the universal fractional maximal operator
$$M^\D_{\al,\mu}f(x):=\sup_{Q\in \D} \frac{1}{\mu(Q)^{1-\frac{\al}{n}}}\int_Q |f|\,d\mu \cdot \chi_Q(x) \qquad 0\leq \al<n.$$
When $\al=0$ we simply write $M_\mu^\D=M_{0,\mu}^\D$.  We have the following theorem concerning the boundedness of $M_{\al,\mu}^\D$.  
\begin{theorem}\label{weightedmax} If $0\leq \al<n$, $1<p\leq \frac{n}{\al}$, and $\frac{1}{q}=\frac{1}{p}-\frac{\al}{n}$, then 
$$\|M^\D_{\al,\mu}f\|_{L^q(\mu)}\leq \Big(1+\frac{p'}{q}\Big)^{1-\frac{\al}{n}}\|f\|_{L^p(\mu)}.$$
\end{theorem}
Before we prove Theorem \ref{weightedmax} we note that the boundedness $M^\D_{\al,\mu}$ from $L^p(\mu)$ to $L^q(\mu)$ was proven in \cite{M}.  However, the constant 
$$\Big(1+\frac{p'}{q}\Big)^{1-\frac{\al}{n}}=(p')^{1-\frac{\al}{n}}\Big(1-\frac{\al}{n}\Big)^{1-\frac{\al}{n}}$$  
seems to be new.  When $\al=0$ we get the well known sharp bound 
$$\|M_\mu^\D f\|_{L^p(\mu)}\leq p'\|f\|_{L^p(\mu)}.$$
We do not know if the constant in Theorem \ref{weightedmax} is sharp for $\al>0$.
\begin{proof} By the standard properties of dyadic cubes we get the inequality
$$\mu(\{x:M^\D_{\al,\mu}f(x)>\lambda\})\leq \left(\frac{1}{\lambda}\int_{\{M_{\al,\mu}^\D f>\lambda\}}|f(x)|\,d\mu(x)\right)^{\frac{n}{n-\al}}.$$
Let $q_0=\frac{n}{n-\al}$ and note that if $q$ is defined as in the statement of Theorem \ref{weightedmax}, then $q>q_0$.  We have
\begin{align*}
\int_{\R^n}M_{\al,\mu}^\D f(x)^q\,d\mu&=q\int_0^\infty \lambda^{q-1}\mu(\{M^\D_{\al,\mu}f(x)>\lambda\})\,d\lambda\\
&\leq q\int_0^\infty \lambda^{q-1} \left(\frac{1}{\lambda}\int_{\{M_{\al,\mu}^\D f(x)>\lambda\}}|f(x)|\,d\mu(x)\right)^{q_0}\,d\lambda\\
&\leq q\left(\int_{\R^n}|f(x)|\Big(\int_0^{M^\D_{\al,\mu}f(x)} \lambda^{q-q_0-1}\,d\lambda\Big)^{1/q_0} \,d\mu(x)\right)^{q_0}\\
&= \frac{q}{q-q_0}\left(\int_{\R^n}|f(x)|M^\D_{\al,\mu}f(x)^{\frac{q}{p'}}\,d\mu(x)\right)^{q_0}\\
&\leq \frac{q}{q-q_0} \|f\|_{L^p(\mu)}^{q_0}\|M_{\al,\mu}^\D f\|_{L^q(\mu)}^{\frac{qq_0}{p'}}
\end{align*}
where in the second inequality we used Minkowski's integral inequality and H\"older's inequality in the last.  Using the fact that $\frac{q}{q-q_0}=1+\frac{p'}{q}$, we see that this yields the desired inequality.
\end{proof}

\section{Main Results}
We now prove rather precise weighted estimates for the sparse operators $T^\Sp$ and $I_\al^\Sp$.  By the Theorems \ref{dyadicbound} and \ref{dyadicboundCZ} we see that these bounds will imply inequalities \eqref{sharpCZ} and \eqref{sharpFrac}.  At the heart of our proof we will simply use the constant $[w]_{A_p}$ in the form of its definition:
$$[w]_{A_p}=\sup_Q\left(\dashint_Q w\,dx\right)\left(\dashint_Q w^{1-p'}\,dx\right)^{p-1}=\sup_Q\frac{w(Q)\sigma(Q)^{p-1}}{|Q|^p}$$
where $\sigma=w^{1-p'}$.  We also use the properties of the family $\{E(Q)\}_{Q\in \Sp}$: disjointness, $E(Q)\subset Q$, and $|Q|\leq 2|E(Q)|$. Our main theorem is the following.
\begin{theorem} \label{MainThmCZ} Suppose $\D$ is a dyadic grid, $\Sp\subset \D$ is a sparse family, $1<p<\infty$, and $w\in A_p$.  Then the following estimate holds 
$$\|T^\Sp\|_{\B(L^p(w))}\leq c_p[w]_{A_p}^{\max(1,\frac{p'}{p})}$$
where
$$c_p=pp' 2^{\max(\frac{p}{p'},\frac{p'}{p})}.$$
\end{theorem}

\begin{proof} Since $T^\Sp$ is a positive operator we may assume $f\geq 0$.  We first consider the case $p\geq 2$, and let $\sigma=w^{1-p'}$.  We will use the well known formulation 
$$\|T^\Sp\|_{\B(L^p(w))}=\|T^\Sp(\,\cdot\,\sigma)\|_{\B(L^p(\sigma),L^p(w))}.$$  
If $g\geq 0$ belongs to $L^{p'}(w)$, then by duality it suffices to estimate 
$$\int_{\R^n}T^\Sp(f\sigma)gw\,dx=\sum_{Q\in \Sp}\dashint_Q f\sigma\,dx\cdot\int_Qgw\,dx.$$
Multiplying and dividing by the precursor to the $A_p$ constant, we have
\begin{align*}
\lefteqn{\sum_{Q\in \Sp}\dashint_Q f\sigma\,dx\cdot\int_Qgw\,dx} \\
&=\sum_{Q\in \Sp}\frac{w(Q)\sigma(Q)^{p-1}}{|Q|^{p}}\frac{|Q|^{p-1}}{w(Q)\sigma(Q)^{p-1}}\int_Q f\sigma\,dx\cdot\int_Qgw\,dx\\
&\leq [w]_{A_p}\sum_{Q\in \Sp}\frac{|Q|^{p-1}}{w(Q)\sigma(Q)^{p-1}}\int_Q f\sigma\,dx\cdot\int_Qgw\,dx\\
&= [w]_{A_p}\sum_{Q\in \Sp}\frac{1}{\sigma(Q)}\int_Q f\sigma\,dx\cdot\frac{1}{w(Q)}\int_Qgw\,dx\cdot|Q|^{p-1}\sigma(Q)^{2-p}\\
&\leq 2^{p-1} [w]_{A_p}\sum_{Q\in \Sp}A_{\sigma}(f,Q)A_w(g,Q)\cdot|E(Q)|^{p-1}\sigma(Q)^{2-p}
\end{align*}
where 
$$A_\sigma(f,Q)=\frac{1}{\sigma(Q)}\int_Qf\sigma\,dx \quad \text{and} \quad A_w(g,Q)=\frac{1}{w(Q)}\int_Qgw\,dx,$$
and in the last inequality we have used $|Q|\leq 2|E(Q)|$.  At this point we have the correct power on the constant $[w]_{A_p}$, so we must estimate the sum in \eqref{est} without using the $A_p$ property of the weight $w$.  
Since $p\geq 2$ and $E(Q)\subset Q$ we have
$$\sigma(Q)^{2-p}\leq \sigma(E(Q))^{2-p},$$
(note: $|E(Q)|\geq|Q|/2>0$ so $\sigma(E(Q))>0$) which in turn yields 
\begin{align}\lefteqn{\int_{\R^n}T^\Sp(f\sigma)gw\,dx}\nonumber\\
&\leq2^{p-1}[w]_{A_p}\sum_{Q\in \Sp}A_{\sigma}(f,Q)A_w(g,Q)\cdot|E(Q)|^{p-1}\sigma(E(Q))^{2-p}.\label{est}
\end{align}
By H\"older's inequality we have
$$|E(Q)|\leq w(E(Q))^{\frac1p}\sigma(E(Q))^{\frac{1}{p'}},$$
so 
\begin{equation}\label{holder}|E(Q)|^{p-1}\sigma(E(Q))^{2-p}\leq \sigma(E(Q))^{\frac1p}w(E(Q))^{\frac{1}{p'}}.\end{equation}
Utilizing inequality \eqref{holder} in the sum in \eqref{est}, followed by a discrete H\"older inequality, followed by the maximal function bounds from Theorem \ref{weightedmax}, we arrive at the desired estimate:
\begin{align*}
\lefteqn{\sum_{Q\in \Sp}A_{\sigma}(f,Q)A_w(g,Q)\cdot |E(Q)|^{p-1}\sigma(E(Q))^{2-p}  }\\
&\leq\sum_{Q\in \Sp}A_{\sigma}(f,Q)A_w(g,Q)\sigma(E(Q))^{\frac1p}w(E(Q))^{\frac{1}{p'}}\\
&\leq \Big(\sum_{Q\in \Sp}A_\sigma(f,Q)^p \sigma(E(Q))\Big)^{\frac{1}{p}} \cdot \Big(\sum_{Q\in \Sp}A_w(g,Q)^{p'} w(E(Q))\Big)^{\frac{1}{p'}}\\
& \leq \|M_\sigma^\D f\|_{L^p(\sigma)}\|M_w^\D g\|_{L^{p'}(w)} \\
&\leq pp' \|f\|_{L^p(\sigma)}\|g\|_{L^{p'}(w)}.
\end{align*}
The case $1<p<2$ follows from duality, since $(T^\Sp)^*=T^\Sp$, we have
$$\|T^\Sp\|_{L^p(w)}=\|T^\Sp\|_{L^{p'}(\sigma)}\leq pp'2^{p'-1} [\sigma]_{A_{p'}}=pp'2^{p'-1}[w]_{A_p}^{\frac{1}{p-1}}.$$
\end{proof}
\begin{theorem} \label{sparsefrac} Suppose $\D$ is a dyadic grid, $\Sp\subset \D$ is a sparse family, $0<\al<n$, $1<p<n/\al$, $\frac1q=\frac1p-\frac{\al}{n}$, $\min\big(\frac{p'}{q},\frac{q}{p'}\big)\leq 1-\frac{\al}{n}$, and $w\in A_{p,q}$.  Then the following estimate holds 

$$\|I_\al^\Sp\|_{\B(L^p(w^p),L^q(w^q))}\leq c_{p,\al}[w]_{A_{p,q}}^{(1-\frac{\al}{n})\max(1,\frac{p'}{q})}$$
where
$$c_{p,\al}=p'\Big(1+\frac{q}{p'}\Big)^{1-\frac{\al}{n}}2^{(1-\frac{\al}{n})\max(\frac{q}{p'},\frac{p'}{q})}.$$
\end{theorem}
Before we prove Theorem \ref{sparsefrac} we remark that we have a somewhat unnatural assumption: \begin{equation}\label{unnat}\min\Big(\frac{p'}{q},\frac{q}{p'}\Big)\leq 1-\frac{\al}{n}.\end{equation}
Because of this we do not obtain the full range of $p$ and $q$ for the fractional integral operator $I_\al$.  We do not know if these techniques can be extended to the full range.  We do point out, however, that inequality \eqref{unnat} is always satisfied when $\al=0$ since $\min(x,1/x)\leq 1$ for $x>0$.  It is for this reason that we do not encounter this obstacle in Theorem \ref{MainThmCZ}. 
\begin{proof}
 Suppose $\frac{p'}{q}\leq 1-\frac{\al}{n}$.  Let $u=w^q$ and $\sigma=w^{-p'}$, so that
$$[w]_{A_{p,q}}=\sup_Q\frac{u(Q)\sigma(Q)^{\frac{q}{p'}}}{|Q|^{1+\frac{q}{p'}}}=[u]_{A_{1+\frac{q}{p'}}}.$$
Define the exponent $r=1+\frac{q}{p'}$ and notice that $r'=1+\frac{p'}{q}$.  We again use the fact 
$$\|I^\Sp_\al\|_{\B(L^p(w^p),L^q(w^q))}=\|I^\Sp_\al(\,\cdot\, \sigma)\|_{\B(L^p(\sigma), L^q(u))}.$$
For $g\in L^{q'}(u)$, we have
$${\int_{\R^n} I^\Sp_\al( f\sigma)gu\,dx=\sum_{Q\in \Sp} |Q|^{\frac{\al}{n}}\dashint_{Q}f\sigma\,dx\cdot \int_{Q} gu\,dx}.$$
Proceeding as above, we multiply and divide by the precursor to the $A_{p,q}$ constant raised to the power $1-\frac{\al}{n}$:
\begin{align*}
\lefteqn{\sum_{Q\in \Sp}\frac{u(Q)^{1-\frac{\al}{n}}\sigma(Q)^{\frac{q}{p'}(1-\frac{\al}{n})}}{|Q|^{(1+\frac{q}{p'})(1-\frac{\al}{n})}}\frac{|Q|^{\frac{q}{p'}(1-\frac{\al}{n})}}{u(Q)^{1-\frac{\al}{n}}\sigma(Q)^{\frac{q}{p'}(1-\frac{\al}{n})}}\int_{Q}f\sigma\,dx\cdot \int_{Q} gu\,dx}\nonumber\\
&\qquad\qquad\leq[w]_{A_{p,q}}^{1-\frac{\al}{n}}\sum_{Q\in \Sp}\frac{1}{\sigma(Q)}\int_{Q} f\sigma\,dx\cdot \frac{1}{u(Q)^{1-\frac{\al}{n}}} \int_{Q} gu\,dx \nonumber\\
&\qquad \qquad \qquad\ \ \cdot |Q|^{\frac{q}{p'}(1-\frac{\al}{n})} \sigma(Q)^{1-\frac{q}{p'}(1-\frac{\al}{n})}\nonumber\\
&\qquad\qquad =[w]_{A_{p,q}}^{1-\frac{\al}{n}}\sum_{Q\in \Sp}A_\sigma(f,Q) A_{u,\al}(g,Q) \cdot |Q|^{\frac{q}{p'}(1-\frac{\al}{n})} \sigma(Q)^{1-\frac{q}{p'}(1-\frac{\al}{n})}
\end{align*}
where $A_{u,\al}(g,Q)=u(Q)^{\frac{\al}{n}-1}\int_Qgu\,dx.$  Once again, we have the correct power on the constant $[w]_{A_{p,q}}$, and so from this point we will not be able to use the $A_{p,q}$ properties of the weight.  Since $\frac{p'}{q}\leq 1-\frac{\al}{n}$ (this is exactly where we use assumption \eqref{unnat}) and $E(Q)\subset Q$ we have
\begin{align*} \sigma(Q)^{1-\frac{q}{p'}(1-\frac{\al}{n})}&\leq \sigma(E(Q))^{1-\frac{q}{p'}(1-\frac{\al}{n})}
\end{align*}
which, along with $|Q|\leq 2|E(Q)|$ yields the bound
\begin{multline}\label{fractest}\lefteqn{\int_{\R^n} I^\Sp_\al(f\sigma)gu\,dx \leq  2^{\frac{q}{p'}(1-\frac{\al}{n})}[w]_{A_{p,q}}^{1-\frac{\al}{n}}}\\ 
\cdot\sum_{Q\in \Sp}A_\sigma(f,Q) A_{\al,u}(g,Q) \cdot  |E(Q)|^{\frac{q}{p'}(1-\frac{\al}{n})} \sigma(E(Q))^{1-\frac{q}{p'}(1-\frac{\al}{n})}.
\end{multline}
Moreover, by H\"older's inequality with $r$ and $r'$ we have
$$|E(Q)|\leq u(E(Q))^{\frac1r}\sigma(E(Q))^{\frac{1}{r'}}.$$
Now notice that
$$\frac{q}{p'}\Big(1-\frac{\al}{n}\Big)=\frac{q}{p'}\Big(\frac1q+\frac{1}{p'}\Big)=\frac{1}{p'}\Big(1+\frac{q}{p'}\Big)=\frac{r}{p'},$$
which yields
\begin{multline}|E(Q)|^{\frac{q}{p'}(1-\frac{\al}{n})}\sigma(E(Q))^{1-\frac{q}{p'}(1-\frac{\al}{n})}=|E(Q)|^{\frac{r}{p'}}\sigma(E(Q))^{1-\frac{r}{p'}} \\
\leq u(E(Q))^{\frac{1}{p'}}\sigma(E(Q))^{\frac{r}{r'p'}}\sigma(E(Q))^{1-\frac{r'}{p}}=u(E(Q))^{\frac{1}{p'}}\sigma(E(Q))^{\frac{1}{p}}.\label{holderest}\end{multline}
Using inequality \eqref{holderest} to estimate the sum in \eqref{fractest} we have
\begin{align*}\lefteqn{\sum_{Q\in \Sp}A_\sigma(f,Q) A_{\al,u}(g,Q) \cdot|E(Q)|^{\frac{q}{p'}(1-\frac{\al}{n})}\sigma(E(Q))^{1-\frac{q}{p'}(1-\frac{\al}{n})}}\\
&\leq \sum_{Q\in \Sp}A_\sigma(f,Q) A_{\al,u}(g,Q) \cdot\sigma(E(Q))^{\frac{1}{p}} u(E(Q))^{\frac{1}{p'}}\\
&\leq \left(\sum_{Q\in \Sp}A_\sigma(f,Q)^p\sigma(E(Q))\right)^{\frac1p} \cdot \left(\sum_{Q\in \Sp}A_{\al,u}(g,Q)^{p'}u(E(Q))\right)^{\frac{1}{p'}}\\
&\leq \|M_\sigma^\D f\|_{L^p(\sigma)}\|M_{\al,u}^\D g\|_{L^{p'}(u)}\\ 
&\leq p'\Big(1+\frac{q}{p'}\Big)^{1-\frac{\al}{n}}\|f\|_{L^p(\sigma)}\|g\|_{L^{q'}(u)}.
\end{align*} 
Where in the last line we have used Theorem \ref{weightedmax} for the boundedness of $M_{\al,u}^\D$ from $L^{q'}(u)$ to $L^{p'}(u)$ (note: $\frac{1}{p'}=\frac{1}{q'}-\frac{\al}{n}$) and $M_\sigma^\D$ from $L^p(\sigma)$ to $L^p(\sigma)$.  The case $\frac{p'}{q}\geq (1-\frac{\al}{n})^{-1}$ again follows from duality; we omit the details.
\end{proof}

\end{document}